\documentclass[12pt,a4paper]{amsart}
\usepackage{amsmath}
\usepackage{amsfonts}
\usepackage{amssymb}
\usepackage{amsthm}
\usepackage[utf8]{inputenc}
\usepackage{mathtools}
\usepackage{array}
\usepackage{fancyhdr}
\usepackage{geometry}
\usepackage{multicol}
\usepackage{enumerate}
\usepackage{stackrel}
\usepackage{color}

\newcommand{\B}{\mathbf{B}}		
\newcommand{\supp}{\operatorname{supp}}
\newcommand{\id}{\operatorname{id}}
\newcommand{\lip}{\operatorname{Lip}}

\renewcommand{\le}{\leqslant}
\renewcommand{\ge}{\geqslant}

\newcommand{\sC}{\operatorname{Cent}}
\newcommand{\sS}{\mathsf{S}}
\newcommand{\sJ}{\mathsf{J}}
\newcommand{\good}{\mathsf{Good}}
\newcommand{\bad}{\mathsf{Bad}}
\newcommand{\fin}{\mathsf{Fin}}

\newcommand{\R}{{\ensuremath{\mathbb R}}}

\newcommand{\dist}{\operatorname{dist}}
\newcommand{\dd}{\phantom{.}\mathrm{d}}

\definecolor{czerwony}{rgb}{0.9, 0.2, 0.1}

\def\Xint#1{\mathchoice
   {\XXint\displaystyle\textstyle{#1}}%
   {\XXint\textstyle\scriptstyle{#1}}%
   {\XXint\scriptstyle\scriptscriptstyle{#1}}%
   {\XXint\scriptscriptstyle\scriptscriptstyle{#1}}%
   \!\int}
\def\XXint#1#2#3{{\setbox0=\hbox{$#1{#2#3}{\int}$}
     \vcenter{\hbox{$#2#3$}}\kern-.5\wd0}}

\def\dashint{\Xint-}

\numberwithin{equation}{section}

\theoremstyle{definition}
\newtheorem{tw}{Theorem}[section]
\newtheorem{lem}[tw]{Lemma}

\newtheorem{stw}[tw]{Proposition}

\newtheorem{rem}[tw]{Remark}

\newtheorem{ex}[tw]{Example}

\newtheorem{cl}[tw]{Claim}

\newtheorem{df}{Definition}[section]

\binoppenalty=\maxdimen
\relpenalty=\maxdimen

\author{Michał Miśkiewicz}
\title{Discrete Reifenberg-type theorem}
\address{Institute of Mathematics, University of Warsaw,\newline Banacha 2, 02-097 Warszawa, Poland}
\email{m.miskiewicz@mimuw.edu.pl}
\thanks{The research has been supported by the NCN grant no. 2012/05/E/ST1/03232 (years 2013-2017).}


\begin{document}

\begin{abstract}
The paper proves that a bound on the averaged Jones' square function of a measure implies an upper bound on the measure. Various types of assumptions on the measure are considered. The theorem is a generalization of a result due to A. Naber and D. Valtorta in connection with measure bounds on the singular set of harmonic maps. 
\end{abstract}

\maketitle

\section{Introduction}

\subsection*{Reifenberg-type theorems}

Classical Reifenberg's theorem states that if a closed set $S \subseteq \R^n$ is well approximated by affine $k$-planes (in the sense of Hausdorff distance) at all balls centered in $S$, then $S$ is bi-H{\"o}lder equivalent with a plane. In this paper we consider approximation in the sense of Hausdorff semi-distance, i.e. sets with holes are allowed. 

The quality of this approximation is measured by Jones' height excess numbers $\beta$. Fix natural numbers $1 \le k < n$ and let $\mu$ be a Radon measure on $\R^n$; the basic example is $\mu = \lambda^k \llcorner S$, where $S$ is a $k$-dimensional set and $\lambda^k$ is the $k$-dimensional Hausdorff measure. We define 
\begin{equation}
\label{eq:betaq}
\beta_{\mu,q}(x,r) = \inf_{V^k} \left( r^{-(k+q)} \int_{\B_r(x)} d^q(y,V^k) \dd \mu(y) \right)^{1/q}. 
\end{equation}
This is the $L^q$ norm of $d(y,V^k)/r$ on $\B_r(x)$ with respect to the measure $r^{-k} \mu$, where $V^k$ is the best affine $k$-plane. 

In order to obtain an upper bound on the measure $\mu$, a uniform bound on $\beta_q(x,r)$ is not sufficient (see Example \ref{ex:flat-snowflake}). The upper bound can follow from a bound on Jones' square function 
\begin{equation}
\label{eq:Jones-square}
J_{\mu,q}(x,r) = \int_0^r \beta_{\mu,q}^2(x,s) \frac{\dd s}{s}. 
\end{equation}
The geometric importance of $J_{\mu,q}$ is illustrated by Example \ref{ex:very-flat-snowflake}. The subscript $\mu$ shall be omitted when it is clear from the context.

\medskip

There are many results concerning the consequences of a bound on Jones' square function. David and Toro \cite{DaTo12} showed that if $S$ satisfies the assumptions of Reifenberg's theorem and $J_{\lambda^k \llcorner S,1}(x,1)$ is uniformly bounded, then the parametrization of $S$ obtained in Reifenberg's theorem is Lipschitz continuous. Azzam and Tolsa \cite{To15}, \cite{AzTo15} characterized rectifiable measures by the condition $J_{\mu,2} < \infty$ $\mu$-a.e., assuming that the upper density is finite $\mu$-a.e. 

This paper is concerned with obtaining upper bounds on the measure $\mu$. In this direction, Naber and Valtorta \cite{NaVa17} proved that there is $\delta(n)>0$ such that if 
\[
r^{-k} \int_{\B_r(x)} J_2(y,r) \dd \mu(y) \le \delta^2 
\]
holds for any ball $\B_r(x) \subseteq \B_2$, then $\mu(\B_1) \le C(n)$. This was proved in two cases: when $\mu$ is a discrete measure and when $\mu = \lambda^k \llcorner S$. In the latter case, the authors also obtained rectifiability of $S$. 

However, it was the discrete version \cite[Th.~3.4]{NaVa17} that was used to obtain an upper bound on the singular set $\lambda^k(\operatorname{Sing} u)$ of a harmonic map $u$ in terms of its Dirichlet energy. A possible application to singular sets of solutions of nonlinear PDEs is one of the main motivations of this paper. 

\subsection*{Basic notation}

The balls centered in $0$ are $\B_r = \B_r(0)$, the measure of $k$-dimensional unit ball is $\omega_k$ and $\lambda \B_r(x) = \B_{\lambda r}(x)$ is the scaled ball. For any set $E$, $\B_r(E)$ is the Minkowski sum $E+\B_r$, i.e. the $r$-neighbourhood of $E$. 

If $\mathcal{S} = \{ \B_j \}$ is a collection of balls, then $\sC \sS$ stands for the set of centers of these balls and $\lambda \mathcal{S} = \{ \lambda \B_j \}$ is the collection of scaled balls with the same centers. We denote the union by 
\[ \bigcup \mathcal{S} = \bigcup_j \B_j. \]

As in \cite{DaTo12}, we use the normalized local Hausdorff distance 
\[ d_{x,r}(E,F) = \frac{1}{r} \dist_H (E \cap \B_r(x), F \cap \B_r(x)), \]
where $\dist_H$ is the standard Hausdorff distance. 

\subsection*{Statement of the main results}

The following is a slightly improved version of Naber and Valtorta's theorem \cite[Th.~3.4]{NaVa17}. The main difference is that the upper bound $J$ is not assumed to be small. Moreover, the theorem holds for any $2 \le q < \infty$. 

\begin{tw}[discrete Reifenberg]
\label{th:discrete-R}
Let $\sS = \{ \B_{r_j}(x_j) \}$ be a collection of disjoint balls in $\B_2$, $\mu = \sum_j \omega_k r_j^k \delta_{x_j}$ be its associated measure and let $\beta_q(x,r)$, $J_q(x,r)$ be defined as in \eqref{eq:betaq}, \eqref{eq:Jones-square}, where $2 \le q < \infty$. Assume that for each ball $\B_r(x) \subseteq \B_2$ we have 
\begin{equation}
\label{eq:mu-flatness}
r^{-k} \int_{\B_r(x)} J_q(y,r) \dd \mu(y) \le J. 
\end{equation}
Then the following estimate holds: 
\begin{equation}
\label{eq:mu-bound}
\mu(\B_1) = \sum_{x_j \in B_1} \omega_k r_j^k \le C(n,q) \cdot \max \left( 1,J^{\frac{q}{q+2}} \right). 
\end{equation}
\end{tw}

The choice of the normalizing constant $\omega_k$ is motivated by the comparison of $\mu$ with $k$-dimensional Hausdorff measure, but has no importance for the theorem. 

\medskip

The proof of Theorem \ref{th:discrete-R} follows the lines of \cite{NaVa17}. This generalization is made possible by relaxing the inductive claim in the construction and carefully keeping track of the constant. 

This observation also leads to other possible extensions, discussed in Section \ref{ch:extensions}. First, Theorem \ref{th:non-discrete} and Remark \ref{rem:non-discrete} generalize the above to measures $\mu$ with controlled upper density, in particular to the case $\mu = \lambda^k \llcorner S$. Second, Theorem \ref{th:weakened} shows that, with minor modifications, the proof applies also with \eqref{eq:mu-flatness} replaced by a weaker assumption 
\[
\dashint_{\B_r(x)} J_q(y,r) \dd \mu(y) \le J. 
\]

\subsection*{Outline of the proof of Theorem \ref{th:discrete-R}}

The main tool is Reifenberg's construction of surfaces $T_0,T_1,T_2,\ldots$ approximating the support of $\mu$. The bound on Jones' square function $J_q$ \eqref{eq:mu-flatness} enables us to prove that this approximation is efficient. There are three key properties that we need: 
\begin{itemize}
\item The total area $|T_i|$ of the approximating surface is estimated from above via $\beta_q$ numbers (see \eqref{eq:key-surface}). 
\item The measures $\mu$ and $\lambda^k \llcorner T_i$ are comparable on (at least some) balls $\B_{r_i}(x)$ centered near $T_i$ (see \eqref{eq:key-comparison}). 
\item The region outside some neighborhood of $T_i$ has small measure $\mu$ (see \eqref{eq:key-excess-set}). 
\end{itemize}
It is intuitive that these three imply some bound on the measure $\mu$. Indeed, once they are derived, we shall see at the end of Section \ref{ch:final-derivation} that the final estimate is an easy consequence. 

\section{Examples}

Reifenberg's theorem states that any $\varepsilon$-Reifenberg flat set is $\alpha$-H\"{o}lder equivalent with a $k$-plane. This leads to finite Hausdorff measure in dimension $k/\alpha$. As $\varepsilon \to 0$, $\alpha$ tends to $1$ and the dimension bound $k/\alpha$ gets arbitrarily close to $k$. The example below shows that under these assumptions this bound cannot be improved. 

\begin{ex}[flat snowflake]
\label{ex:flat-snowflake}
Fix a small angle $\theta$ and consider a modification of the Koch curve (a snowflake): each segment is divided into three segments of equal length and the middle segment is replaced by two segments, each of them at angle $\theta$ to the original segment (the original construction is obtained for $\theta = \pi / 6$). We denote the curve obtained by starting with a unit segment and iterating the above procedure by $K$. 
\end{ex}

If $\theta$ is small, $K$ is $\varepsilon$-Reifenberg flat and $\alpha$-H\"{o}lder equivalent with a segment. For $\theta \approx 0$ we have $\varepsilon \approx \theta \approx 0$ and $\alpha \approx 1$. Still, the Hausdorff dimension of $K$ is greater than $1$. This example shows that Reifenberg's theorem is optimal -- $\varepsilon$-Reifenberg flatness condition does not imply a bound on the $k$-dimensional Hausdorff measure. 

\medskip

Since $\varepsilon$-Reifenberg flatness condition is not enough to imply a bound on the $k$-dimensional Hausdorff measure, we investigate an improved example taken from \cite{DaTo12}. It suggests that the proper hypothesis is a bound on Jones' square function \eqref{eq:Jones-square}. 

\begin{ex}[very flat snowflake]
\label{ex:very-flat-snowflake}
Modify the previous example by taking another angle $\theta_i$ at each stage $i$ of the construction. After $N$ stages we have a curve of length 
\begin{equation*}
\prod_{i=1}^N \frac{2+\frac{1}{\cos \theta_i}}{3} 
= \prod_{i=1}^N \left ( 1 + \frac{1}{6} \theta_i^2 + o(\theta_i^2) \right ). 
\end{equation*}
The product is convergent if and only if the sum $\sum_i \theta_i^2$ converges. The measure $\lambda^1(K)$ of the limit curve can be bounded in terms of this sum. 
\end{ex}
Since the angles $\theta_i$ are comparable with $\beta_q$ numbers taken on the corresponding balls, this shows that indeed the exponent $2$ in the definition of Jones' square function $J_q$ \eqref{eq:Jones-square} is natural. It also suggests that this function can be used to bound the $k$-dimensional measure; indeed, a result of this type was proved in \cite{DaTo12}. In this paper we relax this assumption by concerning a bound on the average $\dashint_{\B_r(x)} J_q(y,r) \dd \mu(y)$ or on $r^{-k} \int_{\B_r(x)} J_q(y,r) \dd \mu(y)$ for each ball $\B_r(x)$. 

\section{Technical constructions}

The tools discussed in this section are well known and most of them are cited from \cite{NaVa17}. Some technical corrections were made in Lemmata \ref{lem:concentration}, \ref{lem:tilt-excess} (counterparts of \cite[4.7,~4.8]{NaVa17}). These corrections come from the fact that the ball $\B_1$ cannot be covered by finitely many balls $\B_\rho(x_i)$ contained in $\B_1$. Thus one is forced to work with a weaker condition $x_i \in \B_1$, in consequence the balls are contained in a slightly larger ball $\B_{1+\rho}$. 

\subsection*{Properties of $\beta$ numbers}

Recall the definitions 
\begin{equation}
\tag{\ref{eq:betaq}}
\beta_q^q(x,r) = \inf_{V^k} r^{-(k+q)} \int_{\B_r(x)} d^q(y,V^k) \dd \mu(y),
\end{equation}
\begin{equation}
\tag{\ref{eq:Jones-square}}
J_{q}(x,r) = \int_0^r \beta_{q}^2(x,s) \frac{\dd s}{s}. 
\end{equation}

Due to the factor $r^{-(k+q)}$ these quantities are scale invariant. Indeed, if $\nu$ is a scaled version of $\mu$, i.e. $\nu(\cdot) = \lambda^{-k} \mu(\lambda \cdot)$, then $\beta_{\nu,q}(0,r) = \beta_{\mu,q}(0,\lambda r)$ and $J_{\nu,q}(0,r) = J_{\mu,q}(0,\lambda r)$. This scaling occurs e.g. if $\nu, \mu$ are discrete measures corresponding to collections of balls $\sS, \lambda \sS$, or $k$-dimensional Hausdorff measure restricted to sets $S, \lambda S$. 

\medskip

First we note the basic continuity property of $\beta_q$. For any $y \in \B_r(x)$ we have $\B_r(x) \subseteq \B_{2r}(y)$ and it follows from the definition that  
\begin{equation}
\label{eq:beta-cont}
\beta_q^q(x,r) \le 2^{k+2} \beta_q^q(y,2r) 
\quad \text{for } y \in \B_r(x). 
\end{equation}
This simple observation leads to an equivalent form of Jones' square function. 

\begin{rem}
\label{rem:flatness-sum}
Fix some $\rho \in (0,1)$ and let $r_\alpha = \rho^\alpha$ for $\alpha = 0,1,2,\ldots$. Then any bound on Jones' square function is (up to a constant depending on $\rho$) equivalent to a bound on 
\begin{equation*}
\sum_{r_\alpha \le 2r} \beta_q^2(x,r_\alpha). 
\end{equation*}
\end{rem}

\begin{proof}
Similarly to \eqref{eq:beta-cont}, we have 
\[ \beta_q^q(x,r_1) \le (r_2/r_1)^{k+q} \beta_q^q(x,r_2) \quad \text{for } r_1 \le r_2. \] 
Take arbitrary $s \in (0,r)$ and choose $\alpha$ such that $\rho^{\alpha+1} \le s < \rho^{\alpha}$. Then 
\begin{alignat*}{3}
c(\rho) \beta_q^2(x,\rho^{\alpha+1}) & \le \beta_q^2(x,s) && \le C(\rho) \beta_q^2(x,\rho^{\alpha}) \\
\text{and} \quad c(\rho) & \le \int_{\rho^{\alpha+1}}^{\rho^{\alpha}} \frac{\dd s}{s} && \le C(\rho), 
\end{alignat*}
which shows the equivalence. 
\end{proof}

\medskip

Denote the auxiliary numbers 
\begin{equation}
\label{eq:delta2}
\delta_q^2(x,r) = r^{-k} \int_{\B_r(x)} \beta_q^2(y,r) \dd \mu(y). 
\end{equation}
Note that assumption \eqref{eq:mu-flatness} together with Remark \ref{rem:flatness-sum} yields a very rough estimate $\delta_q^2(x,r) \le C J$. Moreover, 
\[
\delta_q^2(x_1,r_1) \le C(r_1/r_2) \delta_q^2(x_2,r_2) \qquad \text{if } \B_{r_1}(x_1) \subseteq \B_{r_2}(x_2).  
\]

\medskip

Yet another corollary of \eqref{eq:beta-cont} can be obtained by taking the average over all $y \in \B_r(x)$: 
\[ \beta_q^2(x,r) \le C(k,q) \dashint_{\B_r(x)} \beta_q^2(y,2r) \dd \mu(y). \]
If one assumes a lower bound $\mu(\B_r(x)) \ge \tau(n) M r^k$ (as it will be satisfied in the applications), this can be further estimated by 
\begin{equation}
\label{eq:beta-pointwise-est}
\dashint_{\B_r(x)} \beta_q^2(y,2r) \dd \mu(y)
\le \frac{1}{\tau M r^k} \int_{\B_r(x)} \beta_q^2(y,2r) \dd \mu(y) 
= C(n,tau) M^{-1} \delta_q^2(x,2r). 
\end{equation}

\medskip

Finally, an estimate for $\beta_q^q$ can be obtained by 
\begin{alignat}{3}
\label{eq:beta-nonlinear-est}
\beta_q^q(x,r)
& = \left ( \beta_q^2(x,r) \right )^{q/2} 
&& \lesssim \left ( \dashint_{\B_r(x)} \beta_q^2(y,2r) \dd \mu(y) \right )^{q/2} \nonumber \\
& \lesssim \left ( M^{-1} \delta_q^2(x,2r) \right )^{q/2} 
&& \lesssim M^{-\frac{q}{2}} J^{\frac{q-2}{2}} \delta_q^2(x,2r). 
\end{alignat}

\subsection*{Comparison of $L^q$-best planes via $\beta_q$}

Due to compactness of the Grassmannian $G(k,n)$ and continuity of $d(y,V)$, there exists a $k$-plane minimizing $\int_{\B_r(x)} d^q(y,V) \dd \mu$ (there may be more than one). We choose any of the $L^q$-best planes and denote it by $V(x,r)$. 

We will estimate the distances between the $L^q$-best planes on different balls using $\beta_q$ numbers. More precisely, we want to prove that the distance between $V(x_1,r_1)$ and $V(x_2,r_2)$ is estimated via $\beta_q$ numbers if $r_1,r_2$ are comparable and controlled by $|x_1-x_2|$.

In the case of the standard $\beta_\infty$ numbers this is an elementary geometric problem. As shown by simple examples in \cite{NaVa17}, in case of $\beta_q$ numbers one is forced to assume some kind of Ahlfors-David regularity of the measure $\mu$. Here we use the condition $\tau M r^k \le \mu(\B_r) \le M r$ because we want to study the dependence on $M$ with $\tau(n)$ fixed. 

\begin{lem}
\label{lem:concentration}
There exists $\rho_0(n,\tau)$ such that for $\rho \le \rho_0$ the following holds. If 
\[ \mu(\B_\rho(x)) \le \rho^k \]
holds for all $x \in \B_1$ and $\mu(\B_1) \ge \tau$, then for every affine plane $V \le \R^n$ of dimension $\le k-1$, there exists a point $x \in \B_1$ such that 
\[ d(x,V) > 10\rho, \quad \mu(\B_\rho(x)) \ge C(n,\rho) > 0. \]
\end{lem}

\medskip

Now we can prove the aforementioned tilt-excess result. We denote $\kappa = \frac{1}{1-\rho}$ so that $\kappa \B_\rho(x) \subseteq \kappa \B_1(0)$ for any $x \in \B_1(0)$. 

\begin{lem}
\label{lem:tilt-excess}
Fix $\tau \in (0,1)$ and $\rho(n,\tau)$ as in Lemma \ref{lem:concentration}; denote $\kappa = \frac{1}{1-\rho}$. 
Let $\mu$ be a positive Radon measure. Assume that $\mu(\B_1) \ge \tau M$ and that $\mu(\B_{\rho^2}(y)) \le M \rho^{2k}$ for every $y \in \B_\kappa$. Additionally, let $x \in \B_1$ be such that $\mu(\B_\rho(x)) \ge \tau M \rho^k$. 

Then if $d(x, V(0,\kappa)) \le \rho/2$ or $d(x, V(x,\kappa \rho)) \le \rho/2$, then the distance between the $L^q$-best planes is estimated by 
\[ d_{x,\rho}^q(V(0,\kappa), V(x,\kappa \rho)) \le C(n,q,\rho,\tau) M^{-1} \left( \beta_q^q(0,\kappa) + \beta_q^q(x,\kappa \rho) \right). \]
\end{lem}

\begin{proof}[Sketch of proof]
We assume that $d(x, V(0,\kappa)) \le \rho/2$; in the other case one has to exchange the roles od $V(0,\kappa)$ and $V(x,\kappa \rho)$. For simplicity we take $M=1$, as this amounts to changing $\mu$ by a constant factor. 

The proof follows \cite[Lemma~4.8]{NaVa17}. We choose $k+1$ points $y_0,\ldots,y_k \in \B_\rho(x)$ with $\mu(\B_{\rho^2}(y_i)) \ge c(n,\tau)$. Denote the center of mass by $p_i$ and its projection onto $V(0,\kappa)$ by $p_i'$. We require $p_i'$ to effectively span $V(0,\kappa) \cap \B_\rho(x)$, i.e. 
\[ d(p_{i+1}', \operatorname{span}(p_0',\ldots,p_i')) > 10 \rho^2. \]
This is done by inductive application of Lemma \ref{lem:concentration} and the elementary inequality $|y_i,p_i| \le \rho$. Jensen's inequality yields 
\begin{align*}
d^q(p_i,V(0,\kappa)) & \le C \beta_q^q(0,\kappa), \\
d^q(p_i,V(x,\kappa \rho)) & \le C \beta_q^q(x,\kappa \rho).
\end{align*}

Hence all points $p_i'$ are close to $V(x, \kappa \rho)$. Since these points effectively span $V(0,\kappa) \cap \B_\rho(x)$, this $k$-plane is contained in a small neighborhood of $V(x,\kappa \rho) \cap \B_\rho(x)$. The assumption $d(x, V(0,\kappa)) \le \rho/2$ ensures that the inclusion works both ways (see \cite[Lemma~4.2]{NaVa17}). 
\end{proof}

In the proof of Theorem \ref{th:discrete-R}, the values of $\tau,\rho$ shall be fixed depending only on the dimension $n$. 

\subsection*{Bi-Lipschitz diffeomorphism construction}

Here we introduce the construction later used to obtain the approximating surfaces in the proof of Theorem \ref{th:discrete-R}. 

For some $r>0$, let $\sJ = \{ \B_r(x_i) \}$ be a finite collection of balls such that $\frac 12 \sJ$ is disjoint. For each ball choose a $k$-dimensional affine plane $V_i$ and denote the orthogonal projection onto $V_i$ by $\pi_i$. One can choose a smooth partition of unity $\lambda_i \colon \R^n \to [0,1]$ such that 
\begin{enumerate}
\item $\sum_i \lambda_i \equiv 1$ in $\bigcup 3 \sJ$, 
\item $\lambda_i \equiv 0$ outside $4 \B_{r}(x_i)$ for all $i$, 
\item $||\nabla \lambda_i||_\infty \le C(n)/r$, 
\item if we complete this partition with the smooth function $\psi = 1 - \sum_i \lambda_i$, then $||\nabla \psi||_\infty \le C(n)/r$. 
\end{enumerate}

\begin{df}
\label{df:sigma}
Given $\sJ, \lambda_i, p_i, V_i$ as above, define a smooth function $\sigma \colon \R^n \to \R^n$ by 
\begin{equation*}
\sigma(x) = \psi(x) x + \sum_i \lambda_i(x) \pi_{i} (x). 
\end{equation*}
\end{df}

The function $\sigma$ interpolates between the identity and the projections onto the affine planes $V_i$. Note that $\sigma = \id$ outside of the union $\bigcup 4 \sJ$, as on this region we have $\psi \equiv 1$. On the other hand, if $V_i$ are all close to some $V$, then $\sigma$ is close to the orthogonal projection onto $V$ in the region $\bigcup 3 \sJ$. This will be made precise in Lemma \ref{lem:squash}. 

Lemma \ref{lem:squash} is a modified version of \cite[Lemma~4.12]{NaVa17}. It is essentially a counterpart of the squash lemma used to prove classical Reifenberg's theorem. The crucial additional part of the following is the bi-Lipschitz estimate for $\sigma$ that is quadratic in $\delta_0,\delta_1$; this should be compared to the measure estimate in Example \ref{ex:very-flat-snowflake} and the definition \eqref{eq:Jones-square} of Jones's square function. In order to obtain this quadratic estimate, let us first consider the following geometric fact. 

\begin{lem}
\label{lem:quadratic-estimate}
Let $V_1,V_2$ be two linear $k$-planes and $\pi_1,\pi_2$ be the corresponding orthogonal projections. If $d_{0,1}(V_1,V_2) \le \delta$, then $||\pi_1 \pi_2 - \id||_{V_1 \to V_1} \le C(n) \delta^2$. 
\end{lem}

\begin{proof}
It follows that $V_2 = Q(V_1)$ for some $Q \in \operatorname{SO}(n)$, $||Q-\id|| \lesssim t$. We sketch a proof of this fact, following the proof in \cite[Lemma~4.6]{NaVa17}. Fix some orthonormal bases: $(v_{1,i})_i$ for $V_1$, $(w_{1,j})_{j}$ for $V_1^\perp$. Then applying Gram-Schmidt orthonormalization process (which is Lipschitz) to $(\pi_2(v_{1,i}))_i$ and $(\pi_2^\perp(w_{1,j}))_j$, we get orthonormal bases: $(v_{2,i})_i$ for $V_2$, $(w_{2,j})_{j}$ for $V_2^\perp$. This change of coordinates corresponds to some $Q \in \operatorname{SO}(n)$ as needed. 

Using the fact that the tangent space to $\operatorname{SO}(n)$ at identity is $\operatorname{so}(n)$, we have 
\[
Q = \id + A + O(\delta^2), \quad A^T=-A, \quad ||A|| \le C(n) \delta. 
\]
The projections $\pi_1,\pi_2$ are similiar via $Q$, hence 
\begin{align*}
\pi_2 
& = Q \pi_1 Q^T \\
& = \left ( \id + A + O(\delta^2) \right ) \pi_1 \left ( \id - A + O(\delta^2) \right ) \\
& = \pi_1 + (A \pi_1 - \pi_1 A) - A \pi_1 A + O(\delta^2). 
\end{align*}
The third term is bounded by $||A \pi_1 A|| \le ||A||^2 \le C(n) \delta^2$ and the second disappears after composing with $\pi_1$ and restricting to $V_1$: $\pi_1 (A \pi_1 - \pi_1 A) = 0$ on $V_1$. Hence $||\pi_1 \pi_2 - \id||_{V_1 \to V_1} \le C(n) \delta^2$. 
\end{proof}

The following lemma deals with graphs of functions that are $C^1$ small at scale~$r$. To simplify the notation, we introduce the normalized $C^1$ norm 
\[
||g||_{C^1_r} := r^{-1} ||g||_{\infty} + ||\nabla g||_{\infty}. 
\]

\begin{lem}[squash lemma]
\label{lem:squash}
Fix some ball $\B_{r}(y) \subseteq \R^n$ and a $k$-dimensional affine plane $V$ such that $d(y,V) \le r/2$. Suppose that for all balls $\B_r(x_i) \in \sJ$ centered in $10 \B_{r}(y)$ we have 
\[ d_{x_i,r}(V_i,V) \le \delta_1. \]
Suppose also that $G_0 \subseteq \R^n$ is the graph $G_0 = \{ x+g_0(x): x \in V \} \cap 5 \B_{r}(y)$ of a~small function $g_0 \colon V \to V^\perp$, i.e. $||g_0||_{C^1_r} \le \delta_0$. If $\delta_0 \le 1$ and $\delta_1 \le \delta(n)$, then 
\begin{enumerate}

\item The set $G_1 = \sigma(G_0)$ restricted to $4 \B_r(y)$ is a graph of a function $g_1 \colon V \to V^\perp$ with 
\[ ||g_1||_{C^1_r} \le C(n)(\delta_0+\delta_1). \]
There is ratio $\theta \ge 3 - C(n)(\delta_0+\delta_1)$ such that on each of the balls $\theta \B_r(x_i)$ the previous bound is actually independent of $\delta_0$, i.e. $||g_1||_{C^1_r} \le C(n)\delta_1$. 

\item The map $\sigma \colon G_0 \to G_1$ is a $C^1$ diffeomorphism from $G_0$ to $G_1$ and 
\[
|\sigma(z)-z| \le C(n) (\delta_0+\delta_1) r 
\quad \text{for } z \in G_0. 
\]
Moreover, its bi-Lipschitz constant does not exceed $1 + C(n)(\delta_0^2 + \delta_1^2)$. 

\end{enumerate}
\end{lem}

\begin{proof}
Note that $V_i$ are also close to $V$ on the larger ball: $d_{y,10r}(V_i,V) \le C \delta_1$ for all $i$. For $x \in V$ denote $z = x+g(x)$ and 
\[
h(x) = \sum_i \lambda_i(z) \left ( \pi_i(x+g_0(x)) - x \right ), 
\]
so that 
\begin{align*}
\sigma(x+g_0(x)) 
& = \psi(z) (x+g_0(x)) + \sum_i \lambda_i(z) \pi_i (x+g_0(x)) \\
& = x + \psi(z)g_0(x) + h(x). 
\end{align*}
For simplicity, assume that $0 \in V$. Then we can consider the decomposition of $\sigma$ obtained by projecting onto the linear plane $V$ and its orthogonal complement $V^\perp$: 
\begin{align*}
\sigma(x+g_0(x)) & = \sigma^T(x) + \sigma^\perp(x), \\
\sigma^T(x) & = x + h^T(x), \\
\sigma^\perp(x) & = \psi(z) g_0(x) + h^\perp(x). 
\end{align*}

Now we show that $\sigma^T - \id$ and $\sigma^\perp$ are $C^1_r$-small. Indeed, it is easily checked that $||\pi_i(x+g_0(x))-x||_{C^1_r} \le C \delta_1$ for all $x \in V \cap 5 \B_r(x_i)$ and hence for all $x$ such that $\lambda_i (z) > 0$. Note that this is independent of $\delta_0$, if only $\delta_0 \le 1$. Therefore $||h^T||_{C^1_r}, ||h^\perp||_{C^1_r} \le C \delta_1$. 

The remaining term is estimated by $||\psi(z) g_0(x)||_{C^1_r} \le C \delta_0$, but it vanishes for all $x$ such that $z \in \bigcup 3 \sJ$. 

Thus we obtained 
\[
||\sigma^T - \id||_{C^1_r} \le C \delta_1, \quad ||\sigma^\perp||_{C^1_r} \le C (\delta_0 + \delta_1)
\]
We choose $\delta_1 \le \delta(n)$ small in order to apply the inverse function theorem for $\sigma^T \colon V \to V$. Thus we obtain the inverse function $\phi$ satisfying $||\phi - \id||_{C^1_r} \le C \delta_1$ and $\phi = \id$ outside $\bigcup 4 \sJ$. The inverse enables us to write 
\[
\sigma(x+g_0(x)) = \sigma^T(x) + g_1(\sigma^T(x)), \quad \text{where } g_1(x) = \sigma^\perp(\phi(x)). 
\]
This proves point (1) and the first part of point (2). 

What is left is the estimate for the bi-Lipschitz constant of $\sigma$. To this end, we decompose $\sigma$ in the following way: 
\[
G_0 \ni x+g_0(x) 
\xmapsto{(\id+g_0)^{-1}} x 
\xmapsto{\sigma^T} \sigma^T(x) 
\xmapsto{\id+g_1} \sigma^T(x) + g_1(\sigma^T(x)) 
\in G_1
\]

The Lipschitz constant of the map $V \xrightarrow{\id+g_0} G_0$ is bounded by $\sqrt{1+\delta_0^2}$ and its inverse is a contraction. Similarly, the bi-Lipschitz constant of $V \xrightarrow{\id+g_1} G_1$ is bounded by $\sqrt{1+C(\delta_0^2+\delta_1^2)}$. 

To obtain a quadratic bound for $V \xrightarrow{\sigma^T} V$, we need to improve the estimate $||\nabla h^T||_{\infty} \le C \delta_1$ derived before. To this end, compute 
\begin{align*}
\nabla h^T(x) & = \sum_i \nabla \lambda_i(z) \nabla z \left ( \pi_V \pi_i (x+g_0(x)) - x \right ) \\
& + \sum_i \lambda_i(z) \left ( \pi_V \nabla \pi_i (\id + \nabla g_0(x)) - \id \right )
\end{align*}
In the second sum, the expression in parentheses is $(\pi_V \nabla \pi_i \nabla g_0) + (\pi_V \nabla \pi_i - \id)$. The first term is bounded by $C \delta_0 \delta_1$, while for the second Lemma \ref{lem:quadratic-estimate} implies the bound $C \delta_1^2$. The estimates for the first sum are obtained analogously. Hence $||\nabla h^T||_{\infty} \le C (\delta_0^2+\delta_1^2)$ and the bi-Lipschitz constant of $\sigma^T$ is bounded by $1+C (\delta_0^2+\delta_1^2)$. In consequence, we obtain the bound for $\sigma$ as a composition. 
\end{proof}

We end with a related lemma, which shows that if $G$ is a graph over $V_1$ and $V_1,V_2$ are close, then it is also a graph over 
$V_2$. 

\begin{lem}
\label{lem:changing-planes}
Let $V_1,V_2$ be two affine $k$-planes and $d_{y,r}(V_1,V_2) \le \delta$. Let $G \subseteq \B_r$ be a graph over $V_1$ of a function $g_1$, $||g_1||_{C^1_r} \le \delta$. If $\delta \le \delta(n)$, then $G \cap \theta \B_r$ is also a graph over $V_2$ of a function $g_2$, $||g_2||_{C^1_r} \le C \delta$. The ratio $\theta$ satisfies $1-C\delta < \theta < 1$. 
\end{lem}

\begin{proof}[Sketch of proof]
We follow the proof of Lemma \ref{lem:squash}. The composition 
\[
V_1 \xmapsto{\id+g_1} G \xmapsto{\pi_{V_2}} V_2 
\]
is shown to be a diffeomorphism. If we denote its inverse by $\phi$, then $G \cap \theta \B_r$ is a graph over $V_2$ of $g_2(x) = \phi(x) + g_1(\phi(x))$. 
\end{proof}

\section{Proof of the main theorem}

\subsection*{Induction upwards}

Fix $\tau(n) = 80^{-1} 6^{-n}$, then choose $\rho(n,\tau) \in (0,1)$ according to Lemma \ref{lem:concentration} applied with the value $2^{-k} \tau$ instead of $\tau$, finally denote $\kappa = \frac{1}{1-\rho}$. Without loss of generality we can assume that each of the balls in $\sS$ has radius $r_j = \rho^j$ for some natural $j \ge 1$. Otherwise we exchange each $\B_{r}(x) \in \sS$ for $\B_{r_j}(x)$, where we take $j$ so that $r_j \le r < r_{j-1}$ if $r < \rho$ and $j=1$ if $r \ge \rho$. This only changes the values in \eqref{eq:mu-flatness} and \eqref{eq:mu-bound} by a multiplicative constant. Similarly, we can assume $\mu$ to be supported in $\B_1$, i.e. $\sC \sS \subseteq \B_1$ ($\beta_{\mu,q}$ numbers are monotone in $\mu$). 

Let $\sS^i$ denote those balls that have radius $r_i = \rho^i$; we denote $\sS^{> i}$, $\sS^{< i}$ etc. analogously. We can further assume $\sS$ to be finite. Otherwise we proceed with the truncated collection $\sS^{\le A}$ and its associated measure $\mu^{\le A}$; this measure also satisfies the assumption \eqref{eq:mu-flatness}. If we are able to obtain the claim \eqref{eq:mu-bound} for $\mu^{\le A}$ with a constant independent of $A$, then by passing to the limit $A \to \infty$ we obtain the claim for $\mu$. Thus let us assume that the smallest radius in the collection is $r_A$. 

We focus on proving by induction the following claim: 

\begin{cl}
\label{cl:weak-estimate}
For each $j = A,\ldots,0$ and any ball $\B_{r_j}(x) \subseteq \B_2$ disjoint from $\sC \sS^{\le j}$, 
\[ \mu(\B_{r_j}(x)) \le M r_j^k. \]
At the end of the proof, it shall be clear that $M(n,J) = C(n) \cdot \max(1,J)$ works here. 
\end{cl}

Note that this estimate fails without the additional disjointness assumption, as for any $x \in \sC \sS^i$ and arbitrarily large $j$ we have $\mu(\B_{r_j}(x)) = \omega_k r_i^k$. Still, Claim \ref{cl:weak-estimate} implies our final claim. Indeed, the collection $\sS^{\le 0}$ is empty, thus $\mu(\B_1) \le M$. 

On the other hand, for $j=A$ any ball disjoint from $\sC \sS^{\le A}$ has measure zero, so the claim is trivial. This is the basis for our upwards induction. 

\subsection*{Induction downwards. An outline of the construction}

Here we assume that Claim \ref{cl:weak-estimate} holds for all $x \in \B_1$ and scales $j+1,\ldots,A$ and consider a ball $\B_{r_j}(x)$. For simplicity let us assume $j=0$ and work with the ball $\B_1$ (i.e. the last step of the upwards induction). 

\medskip

We proceed with Reifenberg's construction of coverings of $\sC \sS \cap \B_1$ at all scales $i=0,\ldots,A$. A covering at scale $i$ will consist of the excess set $E^{\le i}$ and collections of balls $\good^i$, $\bad^i$, $\fin^i$, each of radius $r_i$ and centered in $\sC \sS$. The balls $\fin^i$ will be chosen from the collection $\sS^i$ (hence $\mu(\B) = \omega_k r_i^k$ for $\B \in \fin^i$) and the other balls will be separated according to their measure: $\mu(\B) \ge \tau M r_i^k$ for good balls and $\mu(\B) < \tau M r_i^k$ for bad balls. 

As the first step, we define the approximating surface to be 
\[ T_0 = V(0,\kappa) \le \R^n. \]
The covering of $\sC \sS \cap \B_1$ is obtained by just one good ball $\good^0 = \{ \B_1 \}$. Note that if this ball is in fact bad, there is nothing to prove. 

\medskip

The covering will satisfy the following properties: 
\begin{cl}[properties of the covering]
\label{cl:disjoint}
The support of $\mu$ is covered by the collections of balls $\good^i$, $\bad^{\le i}$, $\fin^{\le i}$ and the excess set $E^{<i}$, i.e. 
\[ \sC \sS \subseteq \bigcup \good^i \cup \bigcup \bad^{\le i} \cup \bigcup \fin^{\le i} \cup E^{< i}. \]
The collections $\frac 12 \good^i$, $\frac 12 \bad^{\le i}$, $\frac 12 \fin^{\le i}$ taken together are disjoint. Moreover, the collection $\good^i$ is disjoint from $\sC \sS^{\le i}$. 
\end{cl}

A sequence of surfaces approximating $\sC \sS$ will also be constructed, but it is not used to obtain Claim \ref{cl:disjoint}. 

\subsection*{Excess set}

For each good ball $\B_{r_i}(y) \in \good^i$ we define the excess set 
\[ E(y,r_i) := \B_{r_i}(y) \setminus \B_{r_{i+1}/4}(V(y,\kappa r_i)). \]
This set is exactly what prevents the set $\sC \sS$ from satisfying the uniform Reifenberg condition $\beta_\infty(y,r_i) \le \rho/4$. Its measure will be estimated via Markov inequality later on.

We sum up over all good balls to obtain 
\[ E^{i} := \bigcup_{\good^i} E(y,r_i). \]
We add it to the previous excess sets: $E^{\le i} := E^{\le i-1} \cup E^{i}$. 

Denote the remainder set 
\[ R^{\le i} := \bigcup \bad^{\le i} \cup \bigcup \fin^{\le i} \cup E^{\le i}. \]
The measure of this set can be estimated in a straightforward way, hence we do not need to cover it in the next steps of our inductive construction. 

\subsection*{Construction of the covering}

In order to cover the set $\bigcup \good^i \setminus R^{\le i}$ at scale $r_i$, we first choose the final balls 
\[ \fin^{i+1} := \left\{ \B_{r_{i+1}}(z) : z \in \sC \sS^{i+1} \cap \left( \bigcup \good^i \setminus R^{\le i} \right) \right\}, \]
so that $\fin^{i+1} \subseteq \sS^{i+1}$. Due to Claim \ref{cl:disjoint}, what is left to cover is the set 
\begin{equation}
\label{eq:set-covered}
\tag{$\star$}
\sC \sS^{>i+1} \cap \left( \bigcup \good^i \setminus R^{\le i} \right).
\end{equation}
We choose any maximal $r_{i+1}$-separated subset $\sC \sJ^{i+1}$ of the set \eqref{eq:set-covered} and consider the collection of balls 
\[ \sJ^{i+1} := \{ \B_{r_{i+1}}(z) : z \in \sC \sJ^{i+1} \}. \]
By maximality, the set \eqref{eq:set-covered} is covered by $\bigcup \sJ^{i+1}$. We divide $\sJ^{i+1}$ into two subcollections: 
\begin{align*}
\good^{i+1} & := \left\{ \B \in \sJ^{i+1} : \mu(\B) \ge \tau M r_{i+1}^k \right\}, \\
\bad^{i+1} & := \left\{ \B \in \sJ^{i+1} : \mu(\B) < \tau M r_{i+1}^k \right\}. 
\end{align*}

\begin{proof}[Proof of Claim \ref{cl:disjoint}]
By inductive hypothesis, $R^{\le i}$ covers $\sC \sS^{\le i}$. We covered the rest of $\sC \sS^{i+1}$ by $\fin^{i+1}$ and $\sC \sS^{>i+1}$ by $\good^{i+1}, \bad^{i+1}$, thus we obtained the desired covering. Since the balls in $\sS$ are disjoint and $\sC \sJ^{i+1}$ is an $r_{i+1}$-separated set, the rest of the claim follows. 
\end{proof}

\subsection*{Construction of the approximating surface}

Here we apply the construction from Definition \ref{df:sigma} for the collection of balls $\sJ = \good^{i+1}$. Thus for each ball good $\B_{r_{i+1}}(y_s)$ there is an associated function $\lambda_s$, which together with $\psi$ forms a partition of unity. We choose $V_s$ as the $L^2$-best plane $V(y_s,\kappa r_{i+1})$ on a slightly enlarged ball. This defines the diffeomorphism 
\begin{align*}
\sigma_{i+1}(x) 
& = \psi(x) x + \sum_s \lambda_s(x) \pi_{V_s} (x) 
\end{align*}
and the surface 
\[ T_{i+1} = \sigma_{i+1}(T_i). \]

The construction is now complete. Our aim is to derive three crucial properties \eqref{eq:key-surface}, \eqref{eq:key-comparison}, \eqref{eq:key-excess-set}. Once these are obtained, the final estimate is an easy consequence. First we need some basic properties of the surfaces constructed above. 

\subsection*{Properties of the approximating surface}

\begin{stw}
\label{prop:surface-prop}
\begin{enumerate}[(a)]

\item For $y \in T_i$, 
\[ |\sigma_{i+1}(y) - y| \le \frac{1}{10} r_{i+1}. \]

\item If $\B_{r_{i+1}}(y) \in \good^{i+1}$, then 
\[
|T_{i+1} \cap 5 \B_{r_{i+1}(y)}| \le 10 \cdot \omega_k (5r_{i+1})^k, 
\]

\item $\sigma_{i+1} \colon T_i \to T_{i+1}$ is bi-Lipschitz and for every $\B_{r_{i+1}}(y) \in \good^{i+1}$ its bi-Lipschitz constant on $5 \B_{r_{i+1}}(y)$ is bounded by 
\[
\lip_{i+1} \le 1 + C(n,q,\rho,\tau) M^{-\frac{q+2}{q}} \delta_q^2(y,6r_{i-1}), 
\]
in particular $\lip_{i+1} \le 2^{1/k}$. 

\item If $\B_{r_{i+1}}(y) \in \good^{i+1}$, the surface $T_{i+1}$ is a graph over $V(y,\kappa r_{i+1})$ on $2 \B_{r_{i+1}}(y)$ of a $C^1$ function satisfying 
\[ ||f||_{C^1_{r_{i+1}}} \le C(n,q,\rho,\tau) M^{-\frac{q+2}{q}} \delta_q^2(y,5r_i). \]

\end{enumerate}
\end{stw}

\begin{proof}
In order to derive these, we apply the squash lemma (Lemma \ref{lem:squash}) for a ball $\B_{r_{i+1}(y)} \in \good^{i+1}$. Its center $y$ lies in some $\B_{r_i}(z) \in \good^i$; we let $V := V(z,\kappa r_i)$ be the reference plane. Consider any $y' \in \sC \good^{i+1}$ such that $|y-y'| \le 5r_{i+1}$. Then $y'$ lies in $\B_{2r_i}(z)$ and we may apply Lemma \ref{lem:tilt-excess} (with $2^{-k} \tau$ instead of $\tau$) and obtain 
\begin{align*}
d_{y',r_{i+1}}^2(V(z,2\kappa r_i), V(y',\kappa r_{i+1})) 
& \le C(n,q,\rho,\tau) M^{-\frac{2}{q}} \left( \beta_q^2(y',\kappa r_{i+1}) + \beta_q^2(z,2\kappa r_i) \right) \\
& \le C M^{-\frac{q+2}{q}} \left( \delta_q^2(y',2\kappa r_{i+1}) + \delta_q^2(z,4\kappa r_i) \right) \\
& \le C M^{-\frac{q+2}{q}} \delta_q^2(y,5r_i) \\
& \le C M^{-\frac{q+2}{q}} J 
\end{align*}
Here we used again the pointwise estimate \eqref{eq:beta-pointwise-est} and a bad estimate $\delta_q^2(x,r) \le J$ (the latter shall be refined in the next subsection). We can choose $M \ge C(\tau) J^{\frac{q}{q+2}}$ large enough so that the right-hand side is small. The planes $V(z,2\kappa r_i)$ and $V(z,\kappa r_i)$ are compared in the same way: 
\[
d_{z,r_i}^2(V(z,2\kappa r_i), V(z,\kappa r_i)) 
\le C M^{-\frac{q+2}{q}} \delta^2(y,5r_i) 
\le C M^{-\frac{q+2}{q}} J. 
\]
By the inductive assumption, $T_i$ is a graph over $V(z,\kappa r_i)$ on $2 \B_{r_i}(z)$ hence we can apply Lemma \ref{lem:squash} with 
\[
\delta_1 := \left ( C M^{-\frac{q+2}{q}} \delta_q^2(y,5r_i) \right )^{1/2}, 
\quad
\delta_0 := \left ( C M^{-\frac{q+2}{q}} \delta_q^2(z,5r_{i-1}) \right )^{1/2}. 
\]
Thus we obtain (a) and (b), while (c) follows after an additional estimate on $\delta_0, \delta_1$. 

\medskip

We also obtain an altered version of (d): $T_{i+1}$ is also a graph over $V(z,\kappa r_i)$ on $\theta \B_{r_{i+1}}(y)$ with the desired $C^1$ bound (one can take $\theta = 2.5$). By an application of Lemma \ref{lem:changing-planes}, one can change the plane: $T_{i+1}$ is a graph over $V(y,\kappa r_{i+1})$ on $2 \B_{r_{i+1}}(y)$. This completes the proof of Proposition \ref{prop:surface-prop}. 
\end{proof}

\subsection*{Estimates on the approximating surfaces $T_i$}

By combining the bound for the bi-Lipschitz constant of $\sigma_{i+1} \colon T_i \to T_{i+1}$ in Proposition \ref{prop:surface-prop}c with the elementary estimate $(1+x)^k \le 1+k 2^{k-1} x$ (valid for $x \in [0,1]$), we obtain 
\[
\lip_{i+1}^k \le 1+ C M^{-\frac{q+2}{q}} \sum_s \delta^2(y_s,6r_{i-1}) \chi_{5\B_{r_{i+1}}(y_s)},
\]
where the sum is taken over all balls in $\good^{i+1}$. The measure of $T_{i+1} = \sigma_{i+1}(T_i)$ can be estimated by 
\[ |T_{i+1}| \le \int_{T_i} \lip_{i+1}^k(x) \dd \lambda^k(x). \]
Applying the above estimate and Proposition \ref{prop:surface-prop}b, 
\begin{align*}
|T_{i+1}| 
& \le |T_i| + C M^{-\frac{q+2}{q}} \sum_s |T_i \cap 5 \B_{r_{i+1}}| \delta_q^2(y_s,6r_{i-1}) \\
& \le |T_i| + C M^{-\frac{q+2}{q}} \sum_s \int_{\B_{6r_{i-1}}(y_s)} \beta_q^2(z,6r_{i-1}) \dd \mu(z) \\
& \le |T_i| + C M^{-\frac{q+2}{q}} \int_{\B_2} \beta_q^2(z,6r_{i-1}) \dd \mu(z). 
\end{align*}
In the last line we used the fact that any point $z \in \B_2$ belongs to at most $C(n,\rho)$ balls $\B_{6r_{i-1}}(y_s)$, as the balls $\frac 12 \good^{i+1}$ are disjoint. 

Applying this inductively, we arrive at the following bound: 
\begin{align}
\label{eq:key-surface}
|T_i| & \le |T_0| + C M^{-\frac{q+2}{q}} \sum_{l=0}^{i-1} \int_{\B_2} \beta_q^2(z,6r_{l}) \dd \mu(z) \nonumber \\
& \le \omega_k \left( 1 + C_2(n,q,\rho,\tau) M^{-\frac{q+2}{q}} J \right). 
\end{align}
Here, the bound on the series follows from Remark \ref{rem:flatness-sum}, and equality $|T_0| = \omega_k$ comes from the fact that $T_0$ is a plane. 

\subsection*{Comparison of $\mu$ and $\lambda^k \llcorner T_i$}

Let $\B \in \bad^{i+1} \cup \fin^{i+1}$ be bad or final. In either case, its center $y$ lies in some $\B(z,r_i) \in \good^i$ and $d(y,V(z,\kappa r_i)) \le r_{i+1}/4$, so $T_i$ is a graph over $V(z,\kappa r_i)$ on $\B$. In particular,  
\[ |T_i \cap \B/3| \ge \frac{1}{10} (r_{i+1}/3)^k. \]
Since $|\sigma_{i+1}(y) - y| \le \frac{1}{10} r_{i+1}$ and $\sigma_{i+1}$ has a bi-Lipschitz constant $\lip_{i+1} \le 2^{1/k}$ due to Proposition \ref{prop:surface-prop}, we have 
\begin{align*}
|T_{i+1} \cap \B/2| 
& \ge |T_i \cap \B/3| \cdot \lip_{i+1}^{-k} \\
& \ge 20^{-1} 3^{-k}  r_{i+1}^k.
\end{align*}
By construction, the centers $\sC \good^{>i+1}$ lie outside $\B$, hence $\B/2$ is disjoint with $5 \good^{>i+1}$ and $\sigma_{s} = \id$ on $\B/2$ for $s>i+1$. Therefore 
\[ |T_s \cap \B/2| \ge 20^{-1} 3^{-k} r_{i+1}^k \]
for $s=i,i+1,\ldots$. By definition, $\mu(\B) \le \tau M r_{i+1}^k$ if $\B$ is bad. We choose $M \ge \omega_k/\tau$, so that the same holds if $\B$ is final. Thus we obtain the following comparison estimate 
\begin{equation}
\label{eq:key-comparison}
\mu(\B) \le C_1 \tau M |T_s \cap \B/2| 
\end{equation}
for $\B \in \bad^{i+1} \cup \fin^{i+1}$ and $s=i,i+1,\ldots$. It is essential that the constant $C_1 = 20 \cdot 3^k$ does not depend on $\rho, \tau$. 

\subsection*{Estimates on the excess set}

Since 
\[ E(y,r_i) = \{ x \in \B_{r_i}(y) : d(x,V(y,\kappa r_i)) \ge r_{i+1}/4 \}, \]
Markov inequality yields 
\begin{align*}
\mu(E(y,r_i)) 
& \le \frac{1}{(r_{i+1}/4)^q} \int_{\B_{r_i}(y)} d^q(x,V(y,\kappa r_i)) \dd \mu \\
& \le C(n,q,\rho) r_i^k \beta_q^q(y,\kappa r_i) \\
& \le C(n,q,\rho,\tau) M^{-\frac{q}{2}} J^{\frac{q-2}{2}} r_i^k \delta_q^2(x,2r)
\end{align*}
where in the last line we applied the estimate \eqref{eq:beta-nonlinear-est}. By construction, the balls $\frac{1}{2} \good^i$ are disjoint, hence any point $x \in \R^n$ belongs to at most $C(n)$ of the balls $2 \good^i$. Thus 
\[ \mu(E^{i}) \le C(n,q,\rho,\tau) M^{-\frac{q}{2}} J^{\frac{q-2}{2}} \int_{\B_2} \beta_q^2(x,2r_i) \dd \mu \]
and by summing over $i=0,1,\ldots,A$ we obtain the bound 
\begin{equation}
\label{eq:key-excess-set}
\mu(E^{\le A}) \le C_3(n,q,\rho,\tau) M^{-\frac{q}{2}} J^{\frac{q}{2}}. 
\end{equation}
Here we used again the assumption \eqref{eq:mu-flatness} together with Remark \ref{rem:flatness-sum}. 

\subsection*{Derivation of the bound}
\label{ch:final-derivation}

Here we prove Claim \ref{cl:weak-estimate} using the estimates \eqref{eq:key-excess-set}, \eqref{eq:key-surface}, \eqref{eq:key-comparison}. By construction, the balls $\good^i$ are disjoint from $\sC \sS^{\le i}$. This means that at the $A$-th step of the construction we have $\good^A = \emptyset$, as this collection of balls is disjoint with $\sC \sS$. Therefore $\mu$ is supported in the remainder set: 
\[ \supp \mu \subseteq \bigcup \bad^{\le A} \cup \bigcup \fin^{\le A} \cup E^{\le A}. \]
Recall that the collections $\frac 12 \bad^{\le A}$, $\frac 12 \fin^{\le A}$ are disjoint, so we can use \eqref{eq:key-comparison} for all bad and final balls with $s=A$ to obtain: 
\[ \mu \left( \bigcup \bad^{\le A} \cup \bigcup \fin^{\le A} \right) \le C_1 \tau M |T_A|. \]
Then the surface estimate \eqref{eq:key-surface} yields 
\[ \mu \left( \bad^{\le A} \cup \bigcup \fin^{\le A} \right) \le \omega_k C_1 \tau M (1 + C_2 M^{-\frac{q+2}{q}} J). \]
We add it with the estimate for the excess set \eqref{eq:key-excess-set} and arrive at 
\[ \mu(\B_1) \le M \left( \omega_k C_1 \tau (1 + C_2 M^{-\frac{q+2}{q}} J) + C_3 M^{-\frac{q+2}{2}} J^{\frac{q}{2}} \right). \]
Note that $\tau(n) = 80^{-1} 6^{-n}$ is chosen so that $\omega_k C_1 \tau \le 1/4$. Now we choose the smallest $M$ satisfying 
\[ C_2 M^{-\frac{q+2}{q}} J \le 1, \quad C_3 M^{-\frac{q+2}{2}} J^{\frac{q}{2}} \le \frac 12 \]
and other lower bounds of the form $M \ge C(n,q)$ imposed during the proof; since $\tau(n)$ is fixed, we see that $M = C(n) \cdot \max \left( 1,J^{\frac{q}{q+2}} \right)$. Finally, we are able to estimate 
\[ \mu(\B_1) \le M \left( \frac 14 (1+1) + \frac{1}{2} \right) = M. \]
This ends the proof of Claim \ref{cl:weak-estimate} and Theorem \ref{th:discrete-R}. 

\section{Extentions of the theorem}
\label{ch:extensions}

\subsection*{Generalization to non-discrete measures}

We assume that $S \subseteq \B_2$ is a $\lambda^k$-measurable subset. Here we generalize Theorem \ref{th:discrete-R} to measures of the form $\mu = \lambda^k \llcorner S$, i.e. we show that \eqref{eq:mu-flatness} implies \eqref{eq:mu-bound} in this case as well. This was done as a part of an independent theorem in \cite[Th.~3.3]{NaVa17}, but here we show it is a corollary of Theorem \ref{th:discrete-R}. 

\begin{tw}
\label{th:non-discrete}
Let $S \subseteq \B_2$ be a $\lambda^k$-measurable set and let $\beta_q(x,r)$, $J_q(x,r)$ be defined as in \eqref{eq:betaq}, \eqref{eq:Jones-square} corresponding to the measure $\lambda^k \llcorner S$, where $2 \le q < \infty$. Assume that for each ball $\B_r(x) \subseteq \B_2$ we have 
\begin{equation*}
r^{-k} \int_{S \cap \B_r(x)} J_q(y,r) \dd \lambda^k(y) \le J. 
\end{equation*}
Then for each ball $\B_r(x) \subseteq \B_1$ the following estimate holds: 
\begin{equation*}
\lambda^k(S \cap \B_r(x)) \le C(n,q) \cdot \max \left( 1,J^{\frac{q}{q+2}} \right) \cdot r^k. 
\end{equation*}
\end{tw}

\begin{proof}
It is sufficient to show the claim for the ball $\B_1$. Then for any $\B_r(x) \subseteq \B_1$ we can apply the theorem to the scaled set $S' = \frac 1r(S-x)$, which satisfies the assumptions with the same value of $J$. Thus we obtain 
\[
\lambda^k(S \cap \B_r(x)) = \lambda^k(S' \cap \B_1) \cdot r^k \le C(n) \cdot \max \left( 1,J^{\frac{q}{q+2}} \right) \cdot r^k. 
\]

As a first step we show that $\mu = \lambda^k \llcorner S$ is $\sigma$-finite. Indeed, \eqref{eq:mu-flatness} yields in particular 
\[ \int_{\B_2} J_q(y,2) \dd \mu(y) \le 2^k \cdot J. \]
Choose $t>0$ and define the superlevel set $S_t = \{ y \in S : J_q(y) \ge t J \}$, then $\mu(S_a) \le 2^k /t$ by Markov inequality. On the other hand, the set $S_0 = \{ J_q(y) = 0 \}$ is clearly contained in a $k$-dimensional plane and hence $\mu(S_0) < \infty$. Since 
\[ S = S_0 \cup \bigcup_{j=1}^\infty S_{1/j}, \]
$\mu$ is $\sigma$-finite. We can assume without loss of generality that $\mu$ is finite. Indeed, we can first consider the smaller sets $S_0 \cup S_{1/j}$ instead; since the bound \eqref{eq:mu-bound} depends on $n$ and $q$ only, in the limit we obtain the bound also for $S$. 

Second, we recall the notion of upper $k$-dimensional density 
\[ \Theta^{*k}(S,x) = \limsup_{r \to 0} \frac{\lambda^k(S \cap \B_r(x))}{\omega_k r^k} \]
and its following property \cite{Ma95}: 

\begin{stw}
\label{prop:controlled-density}
Let $S \subseteq \R^n$ be a set with $\lambda^k(S) < \infty$. Then for $\lambda^k$-a.e. $x \in S$, 
\begin{equation}
\label{eq:controlled-density}
2^{-k} \le \Theta^{*k}(S,x) \le 1.
\end{equation}
\end{stw}

Consider the set $S^\star$ of all points $x \in S$ satisfying \eqref{eq:controlled-density}. We can replace $S$ with this possibly smaller set. Since the difference $S \setminus S^\star$ has zero $\lambda^k$ measure, the obtained bound for $S^\star$ holds also for $S$. From now on we assume that all points $x \in S$ satisfy \eqref{eq:controlled-density}. 

For every $x \in S$ choose a radius $r_x \in (0,\rho]$ such that 
\begin{alignat*}{3}
\mu \left ( \frac{1}{10} \B_{r_x}(x) \right ) & \ge 2^{-k-1} \omega_k (r_x/10)^k, & \\
\mu(\B_r(x)) & \le 2 \omega_k r^k & \text{ for all } r \le r_x. 
\end{alignat*}
The set $S$ is covered by balls $\B_{r_x}(x)$ and we can extract a countable Vitali subcovering $\B_j = \B_{r_j}(x_j)$, so that the balls $\frac 15 \B_j$ are disjoint. Choose $p_j$ to be the center of mass of $\frac{1}{10} \B_j$ and define the collection 
\[ \sS := \{ \B_{r_j/10}(p_j) \}. \]
Since $p_j \in \frac{1}{10} \B_j$, we have $\B_{r_j/10}(p_j) \subseteq \frac 15 \B_j$, thus the collection $\sS$ is disjoint. We consider the associated measure 
\[ \nu := \sum_j \omega_k (r_j/10)^k \delta_{p_j}. \]

Our goal now is to reduce the problem for $\mu$ to the already solved problem for the discrete measure $\nu$. We will show that this is possible due to the following comparison estimates: 
\begin{align*}
\mu(\B_1) & \le 2 \cdot 10^k \nu(\B_{1+2\rho}) \\
\beta_{\nu,q}^q(x,s) & \le 2^{k+1} 3^{k+q} \beta_{\mu,q}^q(x,3s).
\end{align*}

For the first estimate, we observe that 
\[ \mu(\B_1) \le \sum_{x_j \in \B_{1+\rho}} \mu(\B_j) \le 2 \cdot 10^k \sum_{x_j \in \B_{1+\rho}} \omega_k (r_j/10)^k \le 2 \cdot 10^k \nu(\B_{1+2\rho}). \]

As for the second, consider a ball $\B_s(x)$ such that $3\B_s(x) \subseteq \B_2$. If there is some $p_j \in \B_s(x)$ with $r_j/10 > 2 s$, then by disjointness of $\sS$ this is the only point from $\supp \nu$ in $\B_s(x)$ and $\beta_{\nu,2}(x,s) = 0$. In the other case, $r_j/10 \le 2 s$ for all $p_j \in \B_s(x)$. Choose an affine $k$-plane $V$. On each $\frac{1}{10}\B_j$ we apply Jensen's inequality for the function $d^q(\cdot,V)$: 
\[ d^q(p_j,V) \le \dashint_{\frac{1}{10} \B_j} d^q(y,V) \dd \mu. \]
This yields 
\[ (r_j/10)^k \omega_k d^q(p_j,V) \le 2^{k+1} \int_{\frac{1}{10} \B_j} d^q(y,V) \dd \mu \]
and hence 
\[ \int_{\B_s(x)} d^q(y,V) \dd \nu 
\le 2^{k+1} \sum_{p_j \in \B_s(x)} \int_{\frac{1}{10} \B_j} d^q(y,V) \dd \mu 
\le 2^{k+1} \int_{\B_{3s}(x)} d^q(y,V) \dd \mu. \]
Taking the infimum on the right-hand side, 
\[ \beta_{\nu,q}^q(x,s) \le 2^{k+1} 3^{k+q} \beta_{\mu,q}^q(x,3s). \]

Therefore the flatness condition \eqref{eq:mu-flatness} is satisfied also for the measure $\nu$ and we obtain our claim by an application of Theorem \ref{th:discrete-R}. To be more precise, one first needs to apply an easy rescaling and covering argument, as one needs to bound $\nu(\B_{1+2\rho})$ instead of $\nu(\B_1)$, and also the obtained estimate works only for balls $\B_s(x)$ such that $3 \B_s(x) \subseteq \B_2$. 
\end{proof}

\begin{rem}
\label{rem:non-discrete}
This proof shows that Theorem \ref{th:discrete-R} actually works of all measures $\mu$ with the covering property resulting from Proposition \ref{prop:controlled-density}. Consider $\mu$ supported in the union of balls $\B_{r_j}(x_j)$, each satisfying 
\begin{alignat*}{3}
\mu \left ( \frac{1}{10} \B_{r_j}(x_j) \right ) & \ge c_\mu (r_j/10)^k, & \nonumber \\
\mu \left ( \B_r(x_j) \right ) & \le C_\mu r_j^k & \text{ for all } r \le r_j. 
\end{alignat*}
In particular, this is satisfied by any $\mu$ such that 
\[
c_\mu \le \Theta^{*k}(\mu,x) \le C_\mu \quad \text{for } \mu-\text{a.e. } x.
\]

If $\mu$ satisfies the assumption \eqref{eq:mu-flatness}, then $\mu(\B_1)$ is bounded as in \eqref{eq:mu-bound}. Naturally, the constant obtained in the final estimate depends on $c_\mu,C_\mu$. 
\end{rem}

\subsection*{Weakened assumptions}

The proof of Theorem \ref{th:discrete-R} applies also with the assumption \eqref{eq:mu-flatness} replaced by $\dashint_\B J_2 \le J$. This means that we consider the integral divided by $\mu(\B_r(x))$ instead of $r^k$. Since there is no a priori upper bound for $\mu$, this assumption is weaker. 

\begin{tw}
\label{th:weakened}
Let $\sS = \{ \B_{r_j}(x_j) \}$ be a collection of disjoint balls in $\B_2$, $\mu = \sum_j \omega_k r_j^k \delta_{x_j}$ be its associated measure and let $\beta_q(x,r)$, $J_q(x,r)$ be defined as in \eqref{eq:betaq}, \eqref{eq:Jones-square}, where $2 \le q < \infty$. Assume that for each ball $\B_r(x) \subseteq \B_2$ we have 
\begin{equation*}
\dashint_{\B_r(x)} J_q(y,r) \dd \mu(y) \le J. 
\end{equation*}
the following estimate holds: 
\begin{equation*}
\mu(\B_1) = \sum_{x_j \in B_1} \omega_k r_j^k \le C(n,q) \cdot \max \left( 1, J^{\frac{q}{2}} \right). 
\end{equation*}
\end{tw}

\begin{proof}[Sketch of proof]
Proceeding as in the proof of Theorem \ref{th:discrete-R}, one obtains the following counterparts of estimates \eqref{eq:key-excess-set}, \eqref{eq:key-surface}: 
\begin{align*}
\mu(E^{\le A}) & \le C_3 J^{\frac{q}{2}}, \\
|T_A| & \le \omega_k \left ( 1 + C_2 M^{-\frac{2}{q}} J \right ).
\end{align*}
The main difference lies in the last step of each estimate, where one needs to bound the integral $\int_{\B_2} J_q(x,r) \dd \mu(x)$. A closer look at the proof shows that in fact an integral over $\B_{1.5}$ is sufficient to bound these quantities (actually, any ball larger than $\B_1$ is sufficient if $\rho$ is small enough). In the case considered in Theorem \ref{th:discrete-R}, this is bounded by $J$; in this case, one has to use the rough estimate $\mu(\B_{1.5}) \le C(n) M$ to obtain 
\[
\int_{\B_{1.5}} J_q(x,r) \dd \mu(x) = \mu(\B_{1.5}) \dashint_{\B_{1.5}} J_q(x,r) \dd \mu \le C(n) M J. 
\]

\medskip

This rough estimate can be derived as follows. Since the collection $\sS = \sS^{\ge 1}$ is disjoint, there are at most $C(n,\rho)$ ball centers $\B_{1.5} \cap \sC \sS^{1}$ and each has measure $\omega_k \rho^k$. The rest of $\B_{1.5}$ can be covered by $C(n,\rho)$ balls of radius $\rho$ disjoint from $\sC \sS^{\le 1}$. By the inductive assumption of Claim \ref{cl:weak-estimate}, each has measure bounded by $M \rho^k$. This yields 
\[
\mu(\B_{1.5}) \le C(n,\rho) \omega_k \rho^k + C(n,\rho) M \rho^k \le C(n,\rho) M. 
\]

\medskip

The proof of the estimate \eqref{eq:key-comparison} carries over without changes: 
\[
\mu(\B) \le C_1 \tau M |T_s \cap \B/2| \quad \text{for } \B \in \bad^{i+1} \cup \fin^{i+1} \text{ and } s \ge i. 
\]
Similarly, these three estimates combined yield 
\[ 
\mu(\B_1) \le M \left( \omega_k C_1 \tau \left( 1 + C_2 M^{-\frac{2}{q}} J \right) + C_3 M^{-1} J^{\frac{q}{2}} \right) 
\]
and the proof works for $M = C(n) \cdot \max \left ( 1, J^{\frac{q}{2}} \right )$. 
\end{proof}

\bibliography{discrete-reifenberg}
\bibliographystyle{acm}

\end{document}